\documentclass[12pt]{amsart}
 \usepackage[T1]{fontenc}
\usepackage{graphicx,mathrsfs,setspace,xypic,fancyhdr,amsmath,comment,fullpage,amssymb}
\usepackage[margin=2.5cm]{geometry}
\newtheorem{thm}{Theorem}[section]
\newtheorem*{thma}{Theorem A}
\newtheorem*{thmb}{Theorem B}
\newtheorem*{corc}{Corollary C}

\newtheorem{cor}[thm]{Corollary}

\newtheorem{lem}[thm]{Lemma}
\newtheorem{prop}[thm]{Proposition}
\theoremstyle{definition}
\newtheorem{defn}[thm]{Definition}
\newtheorem{rem}[thm]{Remark}
\numberwithin{equation}{section}
\makeatletter
\newcommand\xleftrightarrow[2][]{\ext@arrow 0099{\longleftrightarrowfill@}{#1}{#2}}
\def\longleftrightarrowfill@{\arrowfill@\leftarrow\relbar\rightarrow}
\makeatother

\def\mc{\mathcal}

\def\dim{\text{dim}}

\def\P{\mathbb{P}}

\def\R{\mathbb{R}}

\def\Z{\mathbb{Z}}

\def\cE{\mc{E}}
\def\cF{\mc{F}}
\def\cG{\mc{G}}

\def\cL{\mc{L}}

\def\cO{\mc{O}}

\def\cV{\mc{V}}

\begin{document}
\title{Castelnuovo-Mumford Regularity and GV-Sheaves on Irregular Varieties}
\author{Yusuf Mustopa}
\address{Tufts University, Department of Mathematics, Bromfield-Pearson Hall, 503 Boston Avenue, Medford, MA 02155}
\email{Yusuf.Mustopa@tufts.edu}
\begin{abstract}
Inspired by Beauville's recent construction of Ulrich sheaves on abelian surfaces \cite{Bea}, we pose the question of whether a torsion-free sheaf on a polarized smooth projective variety with Castelnuovo-Mumford regularity 1 is a GV (generic vanishing) sheaf, and present evidence that this question is governed by the positivity of curves on generalized Brill-Noether loci.  We prove that it has an affirmative answer for natural polarizations on many well-known irregular surfaces, as well as some polarizations on ruled threefolds over a curve. %Among these are affirmative answers when ${\rm dim}(X)=2$ and $\cO_{X}(1)$ satisfies a mild positivity condition, and also when $X$ is a ruled threefold over a curve and $\cO_{X}(1)$ gives a sufficently positive embedding of $X$ as a scroll in projective space.
\end{abstract}
\maketitle

\section*{Introduction}

Castelnuovo-Mumford regularity measures the homological complexity of a graded module, and is therefore central to understanding coherent sheaves on $\P^N$.  If we restrict our attention to sheaves supported on an \textit{irregular} smooth subvariety $X \subset \P^N,$ the Fourier-Mukai methods developed by Pareschi and Popa in \cite{PP1,PP3,PP4} can also be brought to bear.  In this note, we investigate how the Castelnuovo-Mumford regularity of such sheaves interacts with this "Fourier-Mukai geometry." 

The interactions described by our results depend on whether the generalized Brill-Noether locus associated to $\cO_{X}(1)$ contains a sufficiently positive curve.  It is classically known that such curves exist when ${\rm dim}(X) = 1;$ the picture when ${\rm dim}(X) \geq 2$ is not yet clear, even though some dimension estimates exist for generalized Brill-Noether loci \cite{CP,MPP}.  We expect recent advances in the positivity theory of cycles (e.g.~ \cite{FL}) to play a role in further progress on the questions raised here. 

Turning to details, if $X$ is a smooth projective variety, $\cO_{X}(1)$ is an ample and globally generated line bundle on $X,$ and $\cF$ is a coherent sheaf on $X$ supported in dimension $\geq 1,$ the \textit{CM (Castelnuovo-Mumford) regularity} of $\cF$ on $X$ with respect to $\cO_{X}(1)$ is defined as
\begin{equation}
{\rm reg}_{\cO_{X}(1)}(\cF) := \min\{ m \in \Z : {\forall}i > 0~ H^{i}(\cF(m-i))=0\}
\end{equation}
In addition to measuring the complexity of $\cF,$ CM-regularity also helps measure the positivity of $\cF$; the Castelnuovo-Mumford lemma implies that $\cF(m)$ is globally generated for all $m \geq {\rm reg}_{\cO_{X}(1)}(\cF)$, while $\cF(m)$ may not have any global sections at all if $m < {\rm reg}_{\cO_{X}(1)}(\cF).$  In what follows, a polarization is understood to be a line bundle which is globally generated as well as ample.  

We are interested in the CM-regularity of a coherent sheaf $\cF$ on $X$ in the case $H^{1}(\cO_{X}) \neq 0,$ where the \textit{cohomological support loci}
\begin{equation}
V^{i}(\cF) := \{\alpha \in {\rm Pic}^{0}(X) : H^{i}(\cF \otimes \alpha) \neq 0\}, \hskip5pt 0 \leq i \leq {\rm dim}(X)
\end{equation}
are fundamentally important.  The dimensions of the $V^{i}(\cF)$ are invariant under tensoring by elements of ${\rm Pic}^{0}(X)$; however, the CM regularity of $\cF$ is generally not.  We define the \textit{continuous CM-regularity of $\cF$ with respect to $\cO_{X}(1)$} as
\begin{equation}
\label{eq:cont-cm}
{\rm reg}^{\rm cont}_{\cO_{X}(1)}(\cF) :=  \min\{ m \in \Z : {\forall}i > 0~ V^{i}(\cF(m-i)) \neq {\rm Pic}^{0}(X)\}
\end{equation}

This gives a slightly coarser measure of positivity than CM-regularity; see Lemma \ref{lem:cgg} for a precise statement.  The structure of cohomological support loci was addressed in \cite{GL} and \cite{Ha}, and their connections with positivity were pursued in \cite{De,PP1,PP4}.  A key notion in these papers, as well as this note, is the following:  $\cF$ is said to be a \textit{GV-sheaf} if ${\rm codim}(V^{i}(\cF)) \geq i$ for all $i > 0.$  This property can be viewed as a weak form of positivity in at least one sense; when the Albanese map of $X$ is finite, the stronger condition that ${\rm codim}(V^{i}(\cF)) > i$ for all $i > 0$ (this is known as \textit{M-regularity}) implies that $\cF$ is ample (\cite{De}, Corollary 3.2).

In recent years a special class of sheaves with CM-regularity 0 on a polarized variety of dimension $n \geq 1$ has been intensely studied; this is the class of \textit{Ulrich bundles}, i.e.~ vector bundles whose twist by $-i$ has no cohomology for $1 \leq i \leq n.$  Very recently, Beauville produced a family of rank-2 Ulrich bundles on abelian surfaces via the Serre method \cite{Bea}.  Each rank-2 bundle $\cF$ arising from his construction has the property that $\cF(1)$ is Ulrich; in particular $\cF$ has CM-regularity 1.  It can be verified directly that this $\cF$ is a GV-sheaf.  On a different note, every polarized curve admits an Ulrich bundle $\cE$ (e.g.~ \cite{ESW}) and the vanishing $H^{1}(\cE(-1))=0$ implies that $\cE( -1)$ is a GV-sheaf in this case.  

Since the Ulrich property is an open condition on families of vector bundles, it follows that in general the twist of any Ulrich bundle by $-1$ has continuous CM-regularity equal to 1.  Given that irrational curves and abelian surfaces are the only irregular varieties currently known to admit Ulrich bundles, one can ask if the twist of an Ulrich bundle by $-1$ is always a GV-sheaf.  Our results address the following question, which is broader in scope. 
\medskip

\noindent{$(\star)$ Let $X$ be a smooth projective variety of dimension $n \geq 1$ and let $\cO_{X}(1)$ be an ample and globally generated line bundle on $X.$  If $\cF$ is a torsion-free sheaf on $X$ satisfying ${\rm reg}^{\rm cont}_{\cO_{X}(1)}(\cF) \leq 1,$ is $\cF$ a GV-sheaf?}
\medskip
 
Note that we trivially have an affirmative answer when $H^{1}(\cO_{X})=0.$  One of the motivations for the study of M-regularity and related concepts in \cite{PP1} was to understand subvarieties of abelian varieties via the kind of insight that Castelnuovo-Mumford regularity provides for subvarieties of projective space.  Theorems of this type can be found in \cite{PP1,PP2,PP3} and more recently in \cite{LN}.  Our inquiries point in a different direction, since we are interested in deducing generic vanishing statements from "honest" Castelnuovo-Mumford regularity.

Our first result highlights the importance of the generalized Brill-Noether locus $V^{0}(\cO_{X}(1)).$  Even though the conclusion holds under a weaker hypothesis on $\cF,$ the present phrasing emphasizes the connection with $(\star).$

%Turning to specifics, we have shown that ${\rm codim}(V^{n}(\cF)) \geq n$ whenever ${\rm reg}^{\rm cont}_{\cO_{X}(1)}(\cF) \leq 1$ and $V^{0}(\cO_{X}(1)) \subseteq {\rm Pic}^{0}(X)$ contains a complete intersection curve (Proposition \ref{prop:top-codim}).  To see why the hypothesis on $\cO_{X}(1)$ is not too restrictive, recall that when $X$ is a curve, the locus $V^{0}(\cO_{X}(1))$ is a Brill-Noether locus,    %in the results that we have obtained is that As evidence towards an affirmative answer more generally, we can show that ${\rm codim}(V^{n}(\cF))$ 
 
\begin{thma}
Let $X$ be a smooth projective variety of dimension $n \geq 1$, and let $\cO_{X}(1)$ be an ample and globally generated line bundle on $X$ satisfying the property that $V^{0}(\cO_{X}(1))$ contains a curve $T$ whose numerical class is a proportional to a power of an ample divisor class on ${\rm Pic}^{0}(X).$  If $\cF$ is a torsion-free sheaf on $X$ satisfying ${\rm reg}^{\rm cont}_{\cO_{X}(1)}(\cF) \leq 1,$ then ${\rm codim}(V^{n}(\cF)) \geq n.$ 
\end{thma}

The hypothesis on $\cO_{X}(1)$ is trivially satisfied whenever $V^{0}(\cO_{X}(1))={\rm Pic}^{0}(X).$  An Euler characteristic calculation shows that this holds if $h^{0}(\cO_{X}(1)) > 1$ and $\cO_{X}(1)$ is a GV-sheaf; by Corollary C of \cite{PP4}, any globally generated adjunction of a nef line bundle has this property when $X$ has maximal Albanese dimension.  See Remark \ref{rem:symmetric} for a family of examples satisfying $V^{0}(\cO_{X}(1)) \neq {\rm Pic}^{0}(X).$

The next result follows easily from the proof of Theorem A (which will be discussed momentarily) but its assumption on $\cO_{X}(1)$ is much weaker.

\begin{thmb}
Let $X$ be a smooth projective surface with $h^{1}(\cO_{X}) > 0,$ and let $\cO_{X}(1)$ be an ample and globally generated line bundle on $X$ such that $V^{0}(\cO_{X}(1))$ contains a curve $T$ whose numerical class lies in the interior of the cone of curves of ${\rm Pic}^{0}(X).$  Then $(\star)$ has an affirmative answer for $(X,\cO_{X}(1))$.
\end{thmb}

The following consequence, immediate from Kawamata-Viehweg vanishing, implies that if $\cE$ is any Ulrich bundle on a polarized abelian surface then $\cE(-1)$ is a GV-sheaf.  It is worth noting that Beauville's Ulrich bundles, for which this can be checked explicitly, only account for a hypersurface in the relevant moduli space (Remark 3, \cite{Bea}). 

\begin{corc}
If $X$ is a smooth projective surface and $\cL$ is a nef and big line bundle on $X$ for which $\cO_{X}(1) := \omega_{X} \otimes \cL$ is ample and globally generated, then $(\star)$ has an affirmative answer for $(X,\cO_{X}(1))$.
\end{corc}

We also use Theorem B to obtain a positive answer to $(\star)$ for some natural polarizations on Cartesian and symmetric products of curves (Propositions \ref{prop:prod-curves} and \ref{prop:sym-sq}, respectively) and for all polarizations on some surfaces isogenous to a product of curves (Proposition \ref{prop:isog-prod}).
  %A coherent sheaf $\cF$ on $X$ is said to be a \textit{GV-}sheaf if for all $i \geq 0$ we have
%\begin{equation}
%{\rm codim}\{ \alpha \in {\rm Pic}^{0}(X) : H^{i}(\cF \otimes \alpha) \neq 0\} \geq i
%\end{equation}
%The robustness of the GV condition c at by the fact that unlike global generation, it is stable under tensoring by degree-zero line bundles.

%To indicate why the hypothesis on $\cO_{X}(1)$ in Theorem A is natural, it is helpful to recall that when $X$ is a curve and $c_{1}(\cO_{X}(1))=d$, the locus $V^{0}(\cO_{X}(1))$ is a translate of the Brill-Noether locus $W_{d}(C)$, whose cycle class is a product of ample divisor classes {ACGH}.   In light of recent work on higher-codimensional cycles {DELV, FL} and Brill-Noether loci on higher-dimensional varieties, it is interesting to ask when $V^{0}(\cO_{X}(1))$ satisfies similar positivity conditions in the higher-dimensional case.

%A class of sheaves with CM regularity 0 which has been intensely studied in recent years is the class of Ulrich sheaves, i.e.~ sheaves on a subvariety $X \subseteq \P^N$ whose twist by $-k$ has no cohomology for $1 \leq k \leq {\rm dim}(X).$  This note is inspired by Beauville's recent construction of Ulrich sheaves on abelian surfaces.  It can be verified that for each of his Ulrich sheaves $\cE$, the twist $\cE(-1)$ is a GV-sheaf.  %This is even easier to check for Ulrich sheaves on curves.

The idea behind our proofs of Theorems A and B is to use the curve $T \subseteq V^{0}(\cO_{X}(1))$ to construct a ``large enough" positive cycle in ${\rm Pic}^{0}(X)$ that does not intersect $V^{n}(\cF).$  For Theorem A, this cycle is a Pontryagin product of $T$ (Definition \ref{def:pont-prod}), and we use a calculation from \cite{DELV} to show that its numerical class is proportional to a product of ample divisors (Lemma \ref{lem:delv}).  Our construction suggests a roadmap for settling $(\star)$ in the affirmative:  take a sufficiently positive curve $T$ in ${\rm Pic}^{0}(X)$ and for $1 \leq i \leq n-1,$ prove that the $(i-1)$st Pontryagin product of $T$ intersects every effective cycle of codimension $i-1$ but does not intersect $V^{i}(\cF).$  Up to now we have been unable to carry this out in full generality.  However, we have verified $(\star)$ for some scrollar embeddings of ruled threefolds over a curve (Proposition \ref{thm:3fold}). 
%In the threefold case, it is only necessary to control $V^{2}(\cF).$

\medskip

\noindent
\textbf{Acknowledgments:}  I would like to thank Alex K\"{u}ronya, Rita Pardini and Mihnea Popa for useful discussions and correspondence related to this work, and for valuable comments on a preliminary draft.
\medskip

\section{Preliminaries}

Throughout, we work over an algebraically closed field of characteristic zero.  In what follows, $X$ is a smooth projective variety of dimension $n \geq 1,$ $\cO_{X}(1)$ is an ample and globally generated line bundle on $X,$ and $\cF$ is a coherent sheaf on $X$ supported in dimension $\geq 1.$

\subsection{Continuous CM-regularity}

We discuss the definition (\ref{eq:cont-cm}) in more detail.

\begin{defn}
If $k \in \Z,$ we say $\cF$ is continuously $k-$regular with respect to $\cO_{X}(1)$ if for $1 \leq i \leq n$ we have that $V^{i}(\cF(k-i)) \neq {\rm Pic}^{0}(X).$  
\end{defn}

\begin{lem}
\label{lem:basic-cont-reg}
The following are equivalent:
\begin{itemize}
\item[(i)]{$\cF \otimes \alpha$ is continuously $k-$regular with respect to $\cO_{X}(1)$ for all $\alpha \in {\rm Pic}^{0}(X).$}
\item[(ii)]{$\cF$ is continuously $k-$regular with respect to $\cO_{X}(1)$}.
\item[(iii)]{$\cF \otimes \alpha$ is continuously $k-$regular with respect to $\cO_{X}(1)$ for some $\alpha \in {\rm Pic}^{0}(X).$}
\item[(iv)]{$\cF \otimes \alpha$ is $k-$regular with respect to $\cO_{X}(1)$ in the sense of Castelnuovo-Mumford for some $\alpha \in {\rm Pic}^{0}(X).$ \hfill \qedsymbol}
\end{itemize}  %For every $\alpha \in {\rm Pic}^{0}(X) if and only if $\cF \otimes \alpha$ is continuously $k-$regular for all $\alpha \in {\rm Pic}^{0}(X).$
\end{lem}

\begin{proof}
The implications ${\rm (i)} \Rightarrow {\rm (ii)} \Rightarrow {\rm (iii)} \Rightarrow {\rm (iv)}$ are all immediate.  For ${\rm (iv)} \Rightarrow {\rm (i)},$ observe that for all $i$ the dimension of $V^{i}(\cF(k-i))$ is invariant under tensoring by elements of ${\rm Pic}^{0}(X).$
\end{proof}

The next statement follows from combining (iv) of Lemma \ref{lem:basic-cont-reg} with the corresponding property of Castelnuovo-Mumford regularity.

\begin{cor}
If $\cF$ is continuously $k-$regular, then $\cF$ is continuously $k'-$regular for all $k' \geq k.$ \hfill \qedsymbol
\end{cor}

\begin{defn}
The continuous CM-regularity of $\cF$ with respect to $\cO_{X}(1)$ is the smallest integer $k$ for which $\cF$ is continuously $k-$regular with respect to $\cO_{X}(1);$ we denote it by ${\rm reg}^{\rm cont}_{\cO_{X}(1)}(\cF).$
\end{defn}

We conclude this subsection with evidence for our earlier assertion that continuous CM-regularity helps measure positivity.  Recall that a coherent sheaf $\cF$ on $X$ is \textit{continuously globally generated} if there is a nonempty Zariski-open subset $U \subseteq {\rm Pic}^{0}(X)$ such that the evaluation map
\begin{equation}
\bigoplus_{\alpha \in U}H^{0}(\cF \otimes \alpha) \otimes \alpha^{\vee} \to \cF
\end{equation}
is surjective.  The following fact is immediate from (iv) of Lemma \ref{lem:basic-cont-reg}.

\begin{lem}
\label{lem:cgg}
If $\cF$ is continuously $k-$regular with respect to $\cO_{X}(1),$ then $\cF(k)$ is continuously globally generated.  In particular, $\cF(k)$ is nef. \hfill \qedsymbol
\end{lem}

\begin{rem}
One of the main results in \cite{PP1} is that the M-regularity of a coherent sheaf on an abelian variety implies continuous global generation.  In light of Lemma \ref{lem:cgg} and the line of inquiry suggested by $(\star),$ it is natural to ask whether 0-regularity with respect to $\cO_{X}(1)$ implies M-regularity.   
\end{rem}

%Well-known surfaces of type $(\star)$ include fibrations (e.g.~ ruled surfaces) over irrational curves, symmetric squares of irrational curves, and principally polarized abelian surfaces. 

\subsection{Numerical Cycle Classes and Pontryagin Products}

We now collect some statements on intersections of cycles that will be used in the sequel.  In what follows, $X$ is a smooth projective variety of dimension $n \geq 1.$

%We present the statements on cycles that will be used in the sequel.  In what follows, $X$

\begin{defn}
For $0 \leq k \leq n,$ $N_{k}(X) := N_{k}(X)_{\Z} \otimes \R,$ where $N_{k}(X)_{\Z}$ is the group of algebraic $k-$cycles modulo numerical equivalence.  The numerical dual group $N^{k}(X)$ is the dual of $N_{k}(X).$  
\end{defn}

Since the evaluation pairing $N^{k}(X) \times N_{k}(X) \to \R$ can be identified with an intersection pairing, we view $N^{k}(X)$ as parametrizing numerical classes of cycles of codimension $k.$  The following result on intersections of divisors will be used in the proof of Proposition \ref{prop:prod-curves}.

\begin{lem}
\label{lem:pi-exc}
Let $Y$ and $Z$ be projective varieties of respective dimensions $m$ and $n,$ and let $\pi_{Y}$ and $\pi_{Z}$ be the projections from $Y \times Z$ to $Y$ and $Z$.  Let $D$ be a nef divisor on $Y \times Z,$ and let $H, H_{Y}, H_{Z}$ be ample divisors on $Y \times Z,$ $Y$ and $Z,$ respectively.  Then if $D \cdot H \cdot \pi_{Y}^{\ast}H_{Y}^{m-1} \cdot \pi_{Z}^{\ast}H_{Z}^{n-1} = 0$ we have $D = 0.$
\end{lem}

\begin{proof}
The ampleness of $H_{Y}$ and $H_{Z}$ implies that $\pi_{Y}^{\ast}H_{Y}^{m-1} \cdot \pi_{Z}^{\ast}H_{Z}^{n-1}$ is proportional to the numerical class of a surface in $Y \times Z.$  Since $D$ is nef, it is a limit of ample divisor classes, so $D \cdot \pi_{Y}^{\ast}H_{Y}^{m-1} \cdot \pi_{Z}^{\ast}H_{Z}^{n-1}$ is in the closed cone of curves in $N_{1}(Y \times Z).$  Given that its intersection with $H$ is zero by hypothesis, the fact that $H$ is in the interior of the nef cone of $X$ implies that $D \cdot \pi_{Y}^{\ast}H_{Y}^{m-1} \cdot \pi_{Z}^{\ast}H_{Z}^{n-1}=0.$  

It follows at once that $D \cdot \pi_{Y}^{\ast}H_{Y}^{m} \cdot \pi_{Z}^{\ast}H_{Z}^{n-1}=D \cdot \pi_{Y}^{\ast}H_{Y}^{m-1} \cdot \pi_{Z}^{\ast}H_{Z}^{n}=0.$  Consequently, we have
\begin{equation*}
\begin{tiny}
D \cdot (\pi_{Y}^{\ast}H_{Y} + \pi_{Z}^{\ast}H_{Z})^{m+n-1} = D \cdot \biggl(\binom{m+n-1}{m-1}\pi_{Y}^{\ast}H_{Y}^{m-1} \cdot \pi_{Z}^{\ast}H_{Z}^{n} + \binom{m+n-1}{m}\pi_{Y}^{\ast}H_{Y}^{m}\cdot \pi_{Z}^{\ast}H_{Z}^{n-1} \biggr)=0
\end{tiny}
\end{equation*}

Since $\pi_{Y}^{\ast}H_{Y} + \pi_{Z}^{\ast}H_{Z}$ is ample, the class $(\pi_{Y}^{\ast}H_{Y} + \pi_{Z}^{\ast}H_{Z})^{m+n-1}$ is proportional to the class of a complete intersection curve on $Y \times Z,$ and therefore lies in the interior of the cone of curves in $N_{1}(Y \times Z),$ so that our nef divisor $D$ is $0.$
\end{proof}

Specializing to the case where $X$ is an abelian variety, we now introduce the main tool in the proof of Theorem A.

\begin{defn}
\label{def:pont-prod}
If $X$ is an abelian variety of dimension $g \geq 1,$ $T$ is a subvariety of $X,$ and $1 \leq k \leq g,$ the $k-$th Pontryagin product $T^{\ast (k)}$ of $T$ is the image of $T^{k}$ under the addition map $\sigma^{k} : X^{k} \to X.$

%\begin{equation}
%T^{\ast (k)} := \Biggl\{\sum_{j=1}^{k}x_{j} : x_{1}, \cdots ,x_{k} \in T\Biggr\}
%\end{equation}
\end{defn}

\noindent
Note that when $\dim(T)=1,$ the dimension of $T^{\ast (k)}$ is equal to ${\rm min}\{k,g\}.$  The following result is an immediate consequence of Lemma 1.9(a) in \cite{DELV}.  

\begin{lem}
\label{lem:delv}
Let $X$ be an abelian variety of dimension $g \geq 2,$ and let $\alpha \in N_{k}(X)$ be given, where $1 \leq k \leq g-1.$  If $H$ is an ample divisor class on $X$ and $T \subset X$ is a curve whose numerical class is proportional to $H^{g-1},$ then for $1 \leq k \leq g-1$ the numerical class of $T^{\ast (k)}$ is a positive multiple of $H^{g-k}.$ \hfill \qedsymbol 
\end{lem}

%\noindent
%Note that for any $\cL \in {\rm Pic}^{d}(C),$ tensoring by $\cL$ induces an isomorphism $V^{0}(\cL) \cong W_{d}(C).$  It is m

\begin{prop}
\label{prop:bn-locus}
For $1 \leq d \leq g$ and $\cL \in {\rm Pic}^{d}(C),$ the numerical class of $V^{0}(\cL)$ in $N^{g-d}({\rm Pic}^{0}(C))$ is a positive rational multiple of $\theta^{g-d},$ where $\theta \in N^{1}({\rm Pic}^{0}(C))$ is the numerical class of the theta-divisor, and the image $V^{0}(\cL)^{-}$ of $V^{0}(\cL)$ under inversion on ${\rm Pic}^{0}(C)$ is numerically equivalent to $V^{0}(\cL).$ 
\end{prop}

\begin{proof}
The statement about the numerical class of $V^{0}(\cL)$ follows from Theorem (1.3) on p. 212 of \cite{ACGH}, and the invariance of this class under inversion follows from the invariance of $\theta$ under inversion.
\end{proof}
%\begin{prop}
%$V^{0}(\cL)$ is numerically equivalent to $V^{0}(\cL)^{-}.$
%\end{prop}

\section{Proofs of Theorems A and B}

\noindent
Theorem A is an immediate consequence of the following stronger statement.

\begin{prop}
\label{prop:top-codim}
Let $X$ be a smooth projective variety of dimension $n \geq 2$ and irregularity $q \geq 1,$ and let $\cO_{X}(1)$ be an ample and globally generated line bundle on $X$ such that $V^{0}(\cO_{X}(1))$ contains a curve $T$ algebraically equivalent to a complete intersection of algebraically equivalent ample divisors on ${\rm Pic}^{0}(X)$.  Then if $\cF$ is a torsion-free sheaf on $X$ for which $H^{n}(\cF(1-n))=0,$ we have that ${\rm codim}(V^{n}(\cF)) \geq n.$
\end{prop}

\begin{proof}
We consider the $(n-1)$st Pontryagin product of $T,$ which has the set-theoretic description
\begin{equation}
T^{\ast (n-1)} = \{\alpha_{1} \otimes \cdots \otimes \alpha_{n-1} : \alpha_{1}, \cdots ,\alpha_{n-1} \in T\}
\end{equation}
Fix $\overline{\alpha} := \alpha_{1} \otimes \cdots \otimes \alpha_{n-1} \in T^{\ast (n-1)}.$  Then for each $j$ there exists an effective divisor $D_{j} \in |\cO_{X}(1) \otimes \alpha_{j}|.$  If $D_{\overline{\alpha}}:= \sum_{j}D_{j},$ we have an exact sequence
\begin{equation}
0 \to \cF(1-n) \to \cF \otimes \overline{\alpha} \to \cF|_{D_{\overline{\alpha}}} \otimes \overline{\alpha} \to 0
\end{equation}
It follows at once from our hypothesis on $\cF$ that $H^{n}(\cF \otimes \overline{\alpha}) = 0.$  Letting $\overline{\alpha}$ vary over $T^{\ast (n-1)},$ we conclude that $V^{n}(\cF) \cap T^{\ast (n-1)} = \emptyset.$  

There are now two cases to consider.  If $n \leq q,$ then by Lemma \ref{lem:delv} and our hypothesis on $T,$ we have that $T^{\ast (n-1)}$ is an $(n-1)-$dimensional subvariety of ${\rm Pic}^{0}(X)$ whose cycle class is proportional to a product of ample classes by Lemma \ref{lem:delv}.  In particular, $T^{\ast (n-1)}$ must intersect every subvariety of ${\rm Pic}^{0}(X)$ having codimension at most $n-1.$  It follows that ${\rm codim}V^{n}(\cF) \geq n.$  On the other hand, if $n > q,$ then $T^{\ast (n-1)} = {\rm Pic}^{0}(X),$ so that $V^{n}(\cF) = \emptyset.$
\end{proof}

\begin{rem}
\label{rem:symmetric}
For every $n \geq 2$ there is an $n-$dimensional variety $X$ and a line bundle $\cO_{X}(1)$ on $X$ satisfying the hypothesis of Theorem A for which $V^{0}(\cO_{X}(1)) \neq {\rm Pic}^{0}(X).$  Let $C$ be a smooth projective curve of genus $g \geq 6,$ and let $X:=C^{(n)}$ be its $n-$th symmetric product.  The image of the injective map $f_{n} : {\rm Pic}^{0}(C) \to {\rm Pic}^{0}(C^{n}) \cong {\rm Pic}^{0}(C)^{n}$ defined by $f_{n}(\alpha) = \alpha^{\boxtimes n}$ is the locus invariant under the action of the symmetric group, so $f_{n}$ factors canonically through an isomorphism $f_{(n)} : {\rm Pic}^{0}(C) \to {\rm Pic}^{0}(X).$  If $\cO_{C}(1)$ is an ample and globally generated line bundle on $C$ of degree $d < g-3$, there is a line bundle $\cO_{X}(1)$ on $X$ whose pullback via the quotient map $\pi : C^{n} \to X$ is isomorphic to $\cO_{C}(1)^{\boxtimes n}.$  By the calculations in Section 6.1 of \cite{Iz} we have for all $\alpha \in {\rm Pic}^{0}(C)$ an isomorphism 
\begin{equation}
H^{0}(\cO_{X}(1) \otimes f_{(n)}(\alpha)) \cong {\rm Sym}^{n}H^{0}(\cO_{C}(1) \otimes \alpha).
\end{equation}
It follows that restricting $f_{(n)}$ to $V^{0}(\cO_{C}(1))$ induces an isomorphism $V^{0}(\cO_{C}(1)) \cong V^{0}(\cO_{X}(1)).$  Since the numerical class of $V^{0}(\cO_{C}(1))$ is proportional to $\theta^{g-d}$ by Proposition \ref{prop:bn-locus}, intersecting $V^{0}(\cO_{C}(1))$ with $d-1$ general divisors of numerical class $2\theta$ produces a curve in $V^{0}(\cO_{X}(1))$ whose numerical class is proportional to a power of an ample divisor.
\end{rem}

\noindent
\textit{Proof of Theorem B:}  By Lemma \ref{lem:basic-cont-reg}, we may assume without loss of generality that $\cF$ is 1-regular with respect to $\cO_{X}(1).$  Then $H^{1}(\cF)=0,$ and it is immediate that ${\rm codim}(V^{1}(\cF)) \geq 1.$  

Setting $n=2$ in the proof of Proposition \ref{prop:top-codim}, we can also conclude that $V^{2}(\cF) \cap T = \emptyset.$  Since ${\rm Pic}^{0}(X)$ is an abelian variety, its cones of nef and pseudoeffective divisors are equal, so the interiors of their duals in $N^{q-1}({\rm Pic}^{0}(X))$ coincide as well; consequently $T$ must intersect every codimension-1 subvariety of ${\rm Pic}^{0}(X).$  It follows that ${\rm codim}(V^{2}(\cF)) \geq 2.$ \hfill \qedsymbol

\medskip

%\begin{rem}
%\label{rem:bn-loci}
%While the positivity of the curves in Theorems A and B can be difficult to check, verifying even the existence of a curve in $V^{0}(\cO_{X}(1))$ can be nontrivial.
%\end{rem}

\section{Products of Curves}

We now give some applications of Theorem B.  Observe that the hypothesis of the following result is satisfied when $\cO_{X}(1)$ is the Segre product of ample and globally generated line bundles on $C_1$ and $C_2.$

\begin{prop}
\label{prop:prod-curves}
Let $C_1$ and $C_2$ be smooth projective curves of respective genera $g_{1}$ and $g_{2}$, and let $\cO_{X}(1)$ be an ample and globally generated line bundle on $X := C_1 \times C_2$ such that $|\cO_{X}(1)(-F_{1}-F_{2})|$ is nonempty whenever $F_{1}$ and $F_{2}$ are fibers of the projections from $X$ to $C_{1}$ and $C_{2},$ respectively.  Then $(\star)$ has an affirmative answer for $(X,\cO_{X}(1)).$
\end{prop}

\begin{proof}
It suffices to construct a curve in $V^{0}(\cO_{X}(1))$ satisfying the hypothesis of Theorem B.  For $i=1,2$ let $p_{i} : X \to C_{i}$ and $q_{i} : {\rm Pic}^{0}(X) \cong {\rm Pic}^{0}(C_{1}) \times {\rm Pic}^{0}(C_{2}) \to {\rm Pic}^{0}(C_{i})$ be projection maps, and fix a point $x_{i} \in C_{i}.$  If $F_{i} = p_{i}^{-1}(x_{i}),$ our hypothesis on $\cO_{X}(1)$ implies that for $\alpha_{1} \in V^{0}(\cO_{C_1}(x_{1}))$ and $\alpha_{2} \in V^{0}(\cO_{C_{2}}(x_{2}))$ we have   
\begin{equation}
\label{eq:alphas}
0 \neq H^{0}(\cO_{X}(1)(-F_{1}-F_{2}) \otimes p_{1}^{\ast}\alpha_{1}^{\vee} \otimes p_{2}^{\ast}\alpha_{2}^{\vee}) \subseteq H^{0}(\cO_{X}(1) \otimes  p_{1}^{\ast}\alpha_{1}^{\vee} \otimes p_{2}^{\ast}\alpha_{2}^{\vee})
\end{equation} 
For each $i$ the locus $V^{0}(\cO_{C_i}(x_{i}))$ is the image of $C_{i}$ under an Abel-Jacobi embedding in ${\rm Pic}^{0}(C_{i})$.  The canonical identification of ${\rm Pic}^{0}(C_{1}) \times {\rm Pic}^{0}(C_{2})$ with ${\rm Pic}^{0}(X)$ is induced by the map $(\beta_{1}, \beta_{2}) \mapsto p_{1}^{\ast}\beta_{1} \otimes p_{2}^{\ast}\beta_{2},$ so if $V^{0}(\cO_{C_{i}}(x_{i}))^{-}$ is the image of $V^{0}(\cO_{C_{i}}(x_{i}))$ under inversion on ${\rm Pic}^{0}(C_{i})$ we have from (\ref{eq:alphas}) that 
\begin{equation}
V^{0}(\cO_{C_{1}}(x_{1}))^{-} \times V^{0}(\cO_{C_{2}}(x_{1}))^{-} \cong q_{1}^{-1}V^{0}(\cO_{C_{1}}(x_{1}))^{-} \cap q_{2}^{-1}V^{0}(\cO_{C_{2}}(x_{1}))^{-} \subseteq V^{0}(\cO_{X}(1)).
\end{equation}
Let $H$ be a very ample divisor on ${\rm Pic}^{0}(X).$  Then the numerical cycle class defined by 
\begin{equation}
\gamma := H \cdot q_{1}^{\ast}[V^{0}(\cO_{C_{1}}(x_{1}))^{-}] \cdot q_{2}^{\ast}[V^{0}(\cO_{C_{2}}(x_{2}))^{-}] \in N_{1}(X)
\end{equation}
represents a curve $\Gamma \subseteq V^{0}(\cO_{X}(1))$ obtained by intersecting the surface $q_{1}^{-1}V^{0}(\cO_{C_{1}}(x_{1}))^{-} \cap q_{2}^{-1}V^{0}(\cO_{C_{2}}(x_{1}))^{-}$ with a general divisor linearly equivalent to $H.$  We will be done once we show that $\gamma$ has positive intersection with every nonzero pseudoeffective divisor on ${\rm Pic}^{0}(X).$  

Suppose $D$ is a pseudoeffective divisor on ${\rm Pic}^{0}(X)$ such that $D \cdot \gamma = 0.$  Since ${\rm Pic}^{0}(X)$ is an abelian variety, $D$ is nef.  By Proposition \ref{prop:bn-locus} the classes $[V^{0}(\cO_{C_1}(x_{1}))^{-}]$ and $[V^{0}(\cO_{C_2}(x_{2}))^{-}]$ are proportional to products of ample divisor classes on ${\rm Pic}^{0}(C_{1})$ and ${\rm Pic}^{0}(C_{2}),$ respectively; it then follows from Lemma \ref{lem:pi-exc} that $D = 0.$ 
\end{proof}

Combining the proof of Theorem B with the discussion from Remark \ref{rem:symmetric} yields the following result for symmetric squares of curves.

\begin{prop}
\label{prop:sym-sq}
Let $C$ be a smooth projective curve of genus $g \geq 6,$ and let $\cO_{C}(1)$ be an ample and globally generated line bundle of degree $d < g-3$ on $C.$  If $X = C^{(2)}$ and $\cO_{X}(1)$ is a line bundle on $X$ whose pullback via the quotient map $\pi : C^{2} \to X$ is isomorphic to $\cO_{C}(1)^{\boxtimes 2},$ then $(\star)$ has an affirmative answer for $(X,\cO_{X}(1))$. \hfill \qedsymbol
\end{prop}

\begin{prop}
\label{prop:isog-prod}
Let $C_{1}$ and $C_{2}$  be nonisomorphic smooth projective curves that do not admit nontrivial correspondences, and let $G$ be a finite abelian group which acts freely on $C_{1} \times C_{2}$ and faithfully on $C_1$ and $C_2.$  If $X: = (C_{1} \times C_{2})/G$ and $\cO_{X}(1)$ is an ample and globally generated line bundle on $X,$ then $(\star)$ has an affirmative answer for $(X,\cO_{X}(1)).$ 
\end{prop}

\begin{proof}
It suffices to exhibit $\eta \in {\rm Pic}^{0}(A)$ such that $\cF \otimes \eta$ is a GV-sheaf.  Let $\pi: C_{1} \times C_{2} \to X$ be the quotient map; this is \'{e}tale by our hypothesis on the action of $G.$  By (1.1) of \cite{Pa}, we have that $\pi_{\ast}\cO_{C_{1} \times C_{2}} \cong \cO_{X} \oplus \cG,$ where $\cG$ is a direct sum of nontrivial elements of ${\rm Pic}^{0}(X).$  

By upper semicontinuity and Riemann-Roch, we may then choose $\eta \in {\rm Pic}^{0}(X)$ such that $h^{i}(\cF(j) \otimes \eta \otimes \xi)= h^{i}(\cF(j))$ for $0 \leq i \leq 2, j \in \Z,$ and any $\xi \in {\rm Pic}^{0}(X)$ which is a direct summand of $\cG.$  It follows that $\cF' := \cF \otimes \eta$ is 1-regular with respect to $\cO_{X}(1)$.  We will now show that $\cF'$ (and therefore $\cF$) is a GV-sheaf.

It is enough to show that $\pi^{\ast}\cF'$ (which is torsion-free since $\pi$ is \'{e}tale) is 1-regular with respect to $\pi^{\ast}\cO_{X}(1)$.  Granting this for the moment, the absence of nontrivial correspondences on $C_{1} \times C_{2}$ implies the existence of ample and globally generated line bundles $\cO_{C_{1}}(1),\cO_{C_{2}}(1)$ on $C_{1},C_{2}$ resp.~ such that 
\begin{equation}
\pi^{\ast}\cO_{X}(1) \cong \cO_{C_1}(1) \boxtimes \cO_{C_2}(1)
\end{equation}
Since $\pi^{\ast}\cF'$ is assumed to be 1-regular with respect to $\cO_{C_1}(1) \boxtimes \cO_{C_2}(1),$ Proposition \ref{prop:prod-curves} implies that $\pi^{\ast}\cF'$ is a GV-sheaf on $C_{1} \times C_{2}.$  For any ${\eta}' \in {\rm Pic}^{0}(X),$ we have 
\begin{equation}
H^{i}(\cF' \otimes {\eta}') \subset H^{i}(\cF' \otimes {\eta}' \otimes \pi_{\ast}\cO_{C_{1} \times C_{2}}) \cong H^{i}(\pi^{\ast}\cF' \otimes \pi^{\ast}{\eta}') 
\end{equation}
\noindent
Consequently $\pi^{\ast}(V^{i}(\cE'(-1))) \subseteq V^{i}(\pi^{\ast}\cF').$  Since the map $\pi^{\ast} : {\rm Pic}^{0}(X) \to {\rm Pic}^{0}(C_{1} \times C_{2})$ is finite onto its image, we have
\begin{equation}
{\rm codim}(V^{i}(\cF')) = {\rm codim}(\pi^{\ast}(V^{i}(\cF'))) \geq {\rm codim}(V^{i}(\pi^{\ast}\cF')) \geq i.
\end{equation}
\noindent
so that $\cF'$ is a GV-sheaf as claimed.

To check that $\pi^{\ast}\cF'$ is 1-regular with respect to $\pi^{\ast}\cO_{X}(1),$ note that for $0 \leq i \leq 2$ and $j \in \Z$ we have that
\begin{equation}
H^{i}(\pi^{\ast}\cF'(j)) \cong H^{i}(\cF'(j) \otimes \pi_{\ast}\cO_{C_{1} \times C_{2}}) \cong H^{i}(\cF'(j)) \oplus H^{i}(\cF'(j) \otimes \cG)
\end{equation}
Due the the vanishings guaranteed by our definition of $\cF',$ we have that $\pi^{\ast}\cF'$ is 1-regular as desired.

\end{proof}

\section{Ruled Threefolds}

The purpose of this final section is to prove the following result.  Note that when the curve $C$ has large genus, there are many possibilities for $\cO_{X}(1)$ that are not adjunctions. 

\begin{prop}
\label{thm:3fold}
Let $C$ be a smooth projective curve of genus $g \geq 1,$ let $\cV$ be an ample and globally generated vector bundle of rank 3 on $C,$ and let $X = \P(\cV)$ be the projectivization with structure map $\pi : X \to C.$  If $\cO_{C}(1)$ is a globally generated line bundle of degree $d \geq g$ on $C$ and $\cO_{X}(1) := \pi^{\ast}\cO_{C}(1) \otimes \cO_{\P(\cV)}(1),$ then $(\star)$ has an affirmative answer for $(X,\cO_{X}(1)).$
\end{prop}

\begin{proof}
Let $\cF$ be a torsion-free coherent sheaf which is 1-regular with respect to $\cO_{X}(1).$  It is already clear that ${\rm codim}(V^{1}(\cF)) \geq 1;$ we will show that $V^{i}(\cF) = \emptyset$ for $i=2,3.$  

Fix reduced divisors $D' \in |\cO_{C}(1)|$ and $Z \in |\cO_{\P(\cV)}(1)|,$ and define $Y = \pi^{-1}(D').$  It is immediate that $H:=Y+Z \in |\cO_{X}(1)|.$  Since $d \geq g,$ we have from Riemann-Roch that $V^{0}(\cO_{C}(1)) = {\rm Pic}^{0}(C) \cong {\rm Pic}^{0}(X).$  For each $\alpha \in {\rm Pic}^{0}(C)$ we fix  $D'_{\alpha} \in |\cO_{C}(1) \otimes \alpha|$ and define $D_{\alpha} := \pi^{-1}(D'_{\alpha})+Z$; this is algebraically equivalent to $H.$  The 1-regularity hypothesis on $\cF$ implies that for each $\alpha \in {\rm Pic}^{0}(C)$ and $i=2,3$ we have
\begin{equation}
H^{i}(\cF \otimes \pi^{\ast}\alpha) = H^{i}(\cF \otimes \cO_{X}(D_{\alpha}-H)) \cong H^{i}(\cF|_{D_{\alpha}} \otimes \pi^{\ast}\alpha)
\end{equation}
This vanishes for dimension reasons when $i=3$.  It follows that $V^{3}(\cF) = \emptyset.$  For the remaining case $i=2,$ we consider the Mayer-Vietoris sequence
\begin{equation}
\label{eq:mv-sequence}
0 \to (\cF \otimes \pi^{\ast}\alpha)|_{D_{\alpha}}  \to (\cF \otimes \pi^{\ast}\alpha)|_{\pi^{-1}(D'_{\alpha})} \oplus (\cF \otimes \pi^{\ast}\alpha)|_{Z} \to (\cF \otimes \pi^{\ast}\alpha)|_{\pi^{-1}(D'_{\alpha}) \cap Z} \to 0
\end{equation}
Observe that $\pi^{-1}(D'_{\alpha})$ is the disjoint union of $d$ irreducible components, each of which is isomorphic to $\P^2.$  In particular, $\pi^{-1}(D'_{\alpha}) \cap Z$ is the disjoint union of $d$ smooth irreducible curves, each of which is a line in a copy of $\P^2.$  Since $\cF(1)$ is globally generated and torsion-free, it follows that the torsion-free summand of the restriction of $\cF$ to each irreducible component of $\pi^{-1}(D'_{\alpha}) \cap Z$ is a direct sum of line bundles of degree $-1$ or greater.  Moreover, the restriction of $\pi^{\ast}\alpha$ to each component of $\pi^{-1}(D'_{\alpha})$ is trivial, so $H^{1}((\cF \otimes \pi^{\ast}\alpha)|_{\pi^{-1}(D'_{\alpha}) \cap Z}) = H^{1}(\cF|_{\pi^{-1}(D'_{\alpha}) \cap Z})=0.$  Consequently (\ref{eq:mv-sequence}) implies that $H^{2}((\cF \otimes \pi_{1}^{\ast}\alpha)|_{D_{\alpha}})$ injects into $H^{2}((\cF \otimes \alpha)|_{Y_{\alpha}}) \oplus H^{2}((\cF \otimes \pi_{1}^{\ast}\alpha)|_{Z}).$  We will have shown that $V^{2}(\cF) = \emptyset$ once we verify that $H^{2}((\cF \otimes \alpha)|_{Y_{\alpha}})$ and $H^{2}((\cF \otimes \pi_{1}^{\ast}\alpha)|_{Z})$ are both 0.

For the first summand, it is enough to check that $H^{2}(\cF|_{Y_{\alpha}})=0.$  We already know that $H^{2}(\cF)=0,$ so $H^{2}(\cF|_{Y_{\alpha}})$ injects into $H^{3}(\cF(-Y_{\alpha})).$  Given that $2H$ is linearly equivalent to $(Y_{\alpha}+Z)+(Y_{-\alpha}+Z),$ we have that $H^{3}(\cF(-Y_{\alpha}))$ is a quotient of $H^{3}(\cF(-2)),$ which is 0.

For the second summand, we observe that $H^{2}((\cF \otimes \pi^{\ast}\alpha)|_{Z}) \cong H^{1}(R^{1}\pi_{\ast}(\cF|_{Z}) \otimes \alpha).$  Since $H^{1}(\cF|_{Y_{\alpha} \cap Z})=0$ for all $\alpha$ by our previous argument, it follows from Grauert's theorem that $R^{1}\pi_{1\ast}(\cF|_{Z})=0,$ so $H^{2}((\cF \otimes \pi_{1}^{\ast}\alpha)|_{Z})=0$ as desired.
\end{proof}

\end{document}